\documentclass{amsart} 
\usepackage{amssymb}
\newcommand{\p}{\mathfrak{p}}
\newcommand{\T}{\mathbb{T}}
\newcommand{\C}{\mathbb{C}}
\newcommand{\isom}{\cong}
\newcommand{\tensor}{\otimes}
\newcommand{\ncisom}{\approx}   
\newcommand{\vphi}{\varphi}
\newcommand{\cO}{\mathcal{O}}
\newcommand{\Z}{\mathbb{Z}}
\newcommand{\QQ}{\mathbb{Q}}
\renewcommand{\P}{\mathbb{P}}
\newcommand{\bN}{\mathbf{N}}
\usepackage{xspace}  
\newcommand{\F}{\mathbb{F}}
\newcommand{\m}{\mathfrak{m}}
\DeclareMathOperator{\Tr}{Tr}
\DeclareMathOperator{\tor}{tor}  
\DeclareMathOperator{\disc}{disc}
\DeclareMathOperator{\Norm}{Norm}
\DeclareMathOperator{\Disc}{Disc}
\DeclareMathOperator{\Gal}{Gal}

\DeclareFontEncoding{OT2}{}{} 
  \newcommand{\textcyr}[1]{%
    {\fontencoding{OT2}\fontfamily{wncyr}\fontseries{m}\fontshape{n}%
     \selectfont #1}}
\newcommand{\Sha}{{\mbox{\textcyr{Sh}}}}

\theoremstyle{plain}
\newtheorem{theorem}{Theorem}[section]
\newtheorem{proposition}[theorem]{Proposition}

\newtheorem{conjecture}[theorem]{Conjecture}
\newtheorem{remark}[theorem]{Remark}

\newcommand{\Qbar}{\overline{\Q}}

\newcommand{\Q}{\mathbb{Q}}
\usepackage{hyperref}
\usepackage{multirow}
\newcommand{\n}{\mathfrak{n}}

\newcommand{\ap}[1]{a_{\p_{#1}}}
\newcommand{\Ebar}{\overline{E}}

\newcommand{\fc}{\mathfrak{c}}
\newcommand{\OF}{\cO_F}
\newcommand{\dembele}{Demb\'el{\'e}\xspace}
\usepackage{booktabs}

\input cyracc.def 

\renewcommand{\phi}{\varphi}
\title{A Database of Elliptic Curves over $\Q(\sqrt{5})$---First Report}

\author[Bober et al.]{Jonathan Bober, Alyson Deines, Ariah Klages-Mundt, Benjamin
  LeVeque, R. Andrew Ohana, Ashwath Rabindranath, Paul Sharaba, William
  Stein}
\thanks{This work is supported by NSF grant DMS-0757627, administered
by the American Institute of Mathematics. }
\begin{document}
\begin{abstract} 
  We describe a tabulation of (conjecturally) modular elliptic curves
  over the field $\Q(\sqrt5)$ up to the first elliptic curve of rank $2$. Using
  an efficient implementation of an algorithm of Lassina \dembele
  \cite{dembele:hilbert5}, we computed tables of Hilbert modular
  forms of weight $(2,2)$ over $\Q(\sqrt 5)$, and via a variety of
  methods we constructed corresponding elliptic curves, including
  (again, conjecturally) all elliptic curves over $\Q(\sqrt5)$ that
  have conductor with norm less than or equal to 1831.
\end{abstract} 


\maketitle

\section{Introduction}\label{sec:intro}

\subsection{Elliptic Curves over $\Q$}
Tables of elliptic curves over $\Q$ have been of great value in
mathematical research.  Some of the first such tables were those in
Antwerp IV \cite{antwerpiv}, which included all elliptic curves over
$\Q$ of conductor up to $200$, and also a table of all elliptic curves
with bad reduction only at $2$ and $3$.  

Cremona's book \cite{cremona:algs} gives a detailed description of
algorithms that together output a list of all elliptic curves over
$\Q$ of any given conductor, along with extensive data about each
curve.  The proof that his algorithm outputs {\em all} curves of given
conductor had to wait for the proof of the full modularity theorem in
\cite{breuil-conrad-diamond-taylor}.  Cremona has subsequently
computed tables \cite{cremona:onlinetables} of all elliptic curves
over $\Q$ of conductor up to $220,\!000$, including Mordell-Weil
groups and other extensive data about each curve; he expects to soon
reach his current target, conductor $234,\!446$, which is the smallest
known conductor of a rank $4$ curve.

In a different direction, Stein-Watkins (see \cite{stein-watkins:ants5, bmsw:bulletins}) 
created a table of 136,832,795 elliptic curves over $\Q$ of conductor $\leq 10^8$, and a
table of 11,378,911 elliptic curves over $\Q$ of prime conductor $\leq
10^{10}$. 
There are many curves of large discriminant missing from the
Stein-Watkins tables, since these tables are made by enumerating
curves with relatively small defining equations, and discarding those
of large conductor, rather than systematically finding all curves of
given conductor no matter how large the defining equation.

\subsection{Why $\Q(\sqrt{5})$?}

Like $\Q$, the field $F=\Q(\sqrt{5})$ is a totally real field, and
many of the theorems and ideas about elliptic curves over $\Q$ have
been generalized to totally real fields. As is the case over $\Q$,
there is a notion of modularity of elliptic curves over $F$, and work
of Zhang \cite{zhang:heightsshimura} has extended many results of
Gross-Zagier \cite{gross-zagier} and Kolyvagin
\cite{kolyvagin:mordellweil} to the context of elliptic curves over
totally real fields.

If we order totally real number fields $K$ by the absolute value of
their discriminant, then $F=\Q(\sqrt{5})$ comes next after $\Q$ (the
Minkowski bound implies that $|D_K| \geq (n^n/n!)^2$, where
$n=[K:\Q]$, so if $n\geq 3$ then $|D_K|>20$). That $5$ divides
$\disc(F)=5$ thwarts attempts to easily generalize the
method of Taylor-Wiles to elliptic curves over $F$, which makes
$\Q(\sqrt{5})$ even more interesting. Furthermore $F$ is a PID and
elliptic curves over $F$ admit global minimal models and have well-defined 
notions of minimal discriminants.
The field $F$ also has $31$ CM $j$-invariants, which is far more than
 any other quadratic field (see Section~\ref{sec:cm}).  Letting
 $\vphi=\frac{1+\sqrt{5}}{2}$, we have that the group of units
 $\{\pm 1\} \times \langle \vphi \rangle$ of the
ring $R=\cO_F=\Z[\vphi]$ of integers of $F$ is infinite, leading to
additional complications.  Finally, $F$ has even degree, which
makes certain computations more difficult, as the cohomological
techniques of \cite{greenberg-voight:shimura} are not available.

\subsection{Modularity conjecture}\label{sec:mod}
The following conjecture is open:
\begin{conjecture}[Modularity]\label{conj:mod}
  The set of $L$-functions of elliptic curves over $F$ equals the set
  of $L$-functions associated to cuspidal Hilbert modular newforms
  over $F$ of weight $(2,2)$ with rational Hecke eigenvalues.
\end{conjecture}
Given the progress on modularity theorems initiated by
\cite{wiles:fermat}, we are optimistic that Conjecture~\ref{conj:mod}
will be proved.  {\em We assume Conjecture~\ref{conj:mod} 
for the rest of this paper.} \\

In Section~\ref{sec:hmf} we sketch how to compute Hilbert modular
forms using arithmetic in quaternion algebras. Section~\ref{sec:finding} 
gives numerous methods for finding an elliptic curve corresponding to
a Hilbert modular form. It should be noted that these are the methods 
{\it originally} used to make the tables -- in hindsight, it was discovered
that some of the elliptic curves found using the more specific techniques could 
be found using a better implementation of the sieved enumeration of
Section~\ref{sec:sieve}.
Section~\ref{sec:isoclass} addresses how to find all curves that are
isogenous to a given curve.  
In Section~\ref{sec:cm} we enumerate the CM $j$-invariants in $F$.
We discuss some projects for future work in
Section~\ref{sec:future}.  Finally, Section~\ref{sec:tables} contains
tables that summarize various information about our
dataset \cite{sqrt5data}.

{\bf Acknowledgements.} We would like to thank John Cremona, Noam Elkies, Tom
Fisher, Richard Taylor, John Voight, and the anonymous referee for helpful
conversations.  We would especially like to thank Joanna Gaski for
providing (via the method of Section~\ref{sec:naive}) the explicit
table of elliptic curves that kickstarted this project.
We used {\tt Sage} \cite{sage} extensively throughout this project.

\section{Computing Hilbert modular forms over $F$}\label{sec:hmf}

In Section~\ref{sec:dembele} we sketch \dembele's approach to computing Hilbert
modular forms over $F$, then in Section~\ref{sec:p1rn} we make
some remarks about our fast implementation. 

\subsection{Hilbert modular forms and quaternion 
algebras}\label{sec:dembele}
\dembele \cite{dembele:hilbert5} introduced an algebraic approach via
the Jacquet-Langlands correspondence to computing Hilbert modular
forms of weight $(2,2)$ over $F$.  The Hamiltonian quaternion algebra
$F[i,j,k]$ over $F$ is ramified exactly at the two infinite places,
and contains the maximal order
$$
 S = R\left[\frac{1}{2}(1-\overline{\vphi} i + \vphi j),\,
       \frac{1}{2}(-\overline{\vphi} i + j + \vphi k),\,
       \frac{1}{2}(\vphi i - \overline{\vphi} j + k), \,
       \frac{1}{2}(i + \vphi j - \overline{\vphi} k)\right].
$$
For any nonzero ideal $\n$ in $R=\cO_F$,
let $\P^1(R/\n)$ be the set of equivalence classes of
 column vectors with two coprime entries $a,b \in R/\n$ modulo the
 action of $(R/\n)^*$.  We use the notation $[a:b]$
to denote the equivalence class of 
$\left(\begin{smallmatrix}a\\b\end{smallmatrix}\right)$.
For each prime $\p\mid \n$, we fix a choice of isomorphism
$F[i,j,k]\tensor F_{\p} \ncisom M_2(F_{\p})$, which induces a left
action of $S^*$ on $\P^1(R/\n)$. 
The action of $T_{\p}$, for $p\nmid \n$, is
$T_{\p}([x]) = \sum [\alpha x]$, where the sum is over the classes
$[\alpha]\in S/S^*$ with $N_{\text{red}}(\alpha)=\pi_{\p}$ (reduced quaternion norm),
where $\pi_{\p}$ is a fixed choice of totally positive generator of~$\p$.
The Jacquet-Langlands correspondence implies that
the space of Hilbert modular forms of level $\n$ and weight $(2,2)$ is
noncanonically isomorphic as a module over the Hecke algebra
$$\T=\Z[\,T_\p :  \p \text{ nonzero prime ideal of }R\,]$$ 
to the finite dimensional complex vector space 
$V= \C[S^* \backslash \P^1(R/\n)]$. 

\subsection{Remarks on Computing with $\P^1(R/\n)$}\label{sec:p1rn}

In order to implement the algorithm sketched in
Section~\ref{sec:dembele}, it is critical that we can compute with
$\P^1(R/\n)$ very, very quickly.  For example, to apply the method of
Section~\ref{sec:specialvalues} below, in some cases we have to
compute tens of thousands of Hecke operators.  Thus in this section we
make some additional remarks about this fast
implementation.

When $\n=\p^e$ is a prime power, it is
straightforward to efficiently enumerate representative elements of
$\P^1(R/\p^e)$, since each element $[x:y]$ of $\P^1(R/\p^e)$ has a
unique representative of the form $[1:b]$ or $[a:1]$ with $a$
divisible by $\p$, and these are all distinct.  It is easy 
to put  any $[x:y]$ 
in this canonical form and enumerate the elements of $\P^1(R/\p^e)$, after
choosing a way to enumerate the elements of $R/\p^e$.
An enumeration of $R/\p^e$ is easy to give
once we decide on how to represent $R/\p^e$.

In general, factor $\n = \prod_{i=1}^m \p_i^{e_i}$. 
We have a bijection $\P^1(R/\n) \isom \prod_{i=1}^m
\P^1(R/\p_i^{e_i})$, which allows us to reduce to the prime power
case, at the expense of having to compute the bijection 
$R/\n \isom \prod R/\p_i^{e_i}$.
To this end, we {\em represent elements} of $R/\n$ as
$m$-tuples in $\prod R/\p_i^{e_i}$, thus making
computation of the bijection trivial.  

To minimize dynamic memory allocation, thus speeding up the
code by an order of magnitude, in the implementation
we make some arbitrary
bounds; this is not a serious constraint, since the linear algebra
needed to isolate eigenforms for levels beyond this bound is
prohibitive. We assume $m\leq 16$ and each individual 
$p_i^{e_i}\leq 2^{31}$, where $p_i$ is the residue characteristic of $\p_i$.  
In all cases, we represent an
element of $R/\p_i^{e_i}$ as a pair of  64-bit integers, and
represent an element of $R/\n$ as an array of 16 pairs of 64-bit
integers. We use this representation in all cases, even if $\n$ is
divisible by less than $16$ primes; the gain in speed coming
from avoiding dynamic memory allocation more than compensates for the
wasted memory.

Let $\p^e$ be one of the prime power factors of $\n$, and let $p$ be the residue
characteristic of $\p$. We have one of the following cases:
\begin{itemize}
\item $\p$ splits in $R$; then $R/\p\isom \Z/p\Z$ and we represent elements
of $R/\p^e$ as pairs $(a,0)$ mod $p^e$ with the usual addition and multiplication
in the first factor.
\item $\p$ is inert in $R$; then $R/\p^e\isom (\Z/p^e\Z)[x]/(x^2-x-1)$,
and we represent elements by pairs $(a,b) \in \Z/p^e\Z$ with multiplication
$$(a,b)(c,d) = (ac+bd,ad+bd+bc) \mod p^e.$$
\item $\p$ is ramified and $e=2f$ is even; this is exactly the
same as the case when $\p$ is inert but with $e$ replaced by $f$,
since $R/\p^eR \isom (\Z/p^{f}\Z)[x]/(x^2-x-1)$.
\item $\p$ is ramified (so $p=5$) and $e=2f-1$ is odd; the ring $A=R/\p^e$
  is trickier than the rest, because it is {\em not} of the form
  $\Z[x]/(m,g)$ where $m\in\Z$ and $g \in\Z[x]$.  We have $A \ncisom
  (\Z/5^f\Z)[x]/(x^2-5, 5^{f-1}x)$, and represent elements of $A$ as
  pairs $(a,b) \in (\Z/5^f)\times (\Z/5^{f-1}\Z)$, with arithmetic
  given by
\begin{align*}
(a,b) + (c,d) &= (a+c \mod 5^f,\,\,\, b+d \mod 5^{f-1})\\
(a,b)\cdot (c,d) &= (ac+5bd \mod 5^f,\,\,\, ad+bc \mod 5^{f-1}).
\end{align*}
 We find that $ \varphi \in R \mapsto (1/2,1/2)$.
\end{itemize}

\section{Strategies for finding an elliptic curve attached to a Hilbert modular form}\label{sec:finding}
In this section we describe various strategies to find an elliptic curve
 associated to each of the Hilbert modular forms computed in Section 2.
Let $f$ be a rational cuspidal Hilbert newform of weight $(2,2)$ as
in Section~\ref{sec:hmf}.  According to Conjecture~\ref{conj:mod},
there is some elliptic curve $E_f$ over $F$ such that $L(f,s) =
L(E_f,s)$.  (Note that $E_f$ is only well defined up to isogeny.)
Unlike the case for elliptic curves over $\Q$ (see \cite{cremona:algs}), 
there seems to be no known {\em efficient} direct algorithm to find $E_f$.
 Nonetheless, there are several approaches coming from various directions, 
which are each efficient in some cases.

Everywhere below, we continue to assume that Conjecture~\ref{conj:mod}
is true and assume that we have computed (as in Section~\ref{sec:hmf})
the Hecke eigenvalues $a_\p\in \Z$ of all rational Hilbert newforms of
some level $\n$, for $\Norm(\p)\leq B$ a good prime, where $B$ is
large enough to distinguish newforms. In some cases we will need far
more $a_{\p}$ in order to compute with the $L$-function attached to a
newform.  We will also need the $a_{\p}$ for bad $\p$ in a few cases,
which we obtain using the functional equation for the $L$-function (as
an application of Dokchitser's algorithm \cite{dokchitser:lfun}).

We define the {\em norm conductor} of an elliptic curve over $F$ to be
the absolute norm of the conductor ideal of the curve.

In Section~\ref{sec:naive} we give a very simple enumeration method
for finding curves, then in Section~\ref{sec:sieve} we refine it by
taking into account point counts modulo primes; together, these two
methods found a substantial fraction of our curves.
Sections~\ref{sec:torsion} and \ref{sec:congfam} describe methods for
searching in certain families of curves, e.g., curves with a torsion
point of given order or curves with a given irreducible mod $\ell$
Galois representation.  Section~\ref{sec:twisting} is about how to
find all twists of a curve with bounded norm conductor.  In
Section~\ref{sec:cremona-lingham} we mention the Cremona-Lingham
algorithm, which relies on computing all $S$-integral points on many
auxiliary curves.  Finally, Section~\ref{sec:specialvalues} explains
in detail an algorithm of \dembele{} that uses  explicit
computations with special values of $L$-functions to find curves.

\subsection{Extremely naive enumeration}\label{sec:naive}

The most naive strategy is to systematically enumerate 
 elliptic curves $E: y^2 = x^3 + ax + b$, with $a,b\in R$,
and for each $E$, to compute $a_\p(E)$ for $\p$ not dividing $\Disc(E)$
 by counting points on $E$ reduced modulo $\p$.  If all the $a_\p(E)$ match
 with those of the input newform $f$ up to the bound $B$, we then compute
 the conductor $\n_E$, and if it equals $\n$, we conclude from the sufficient 
largeness of $B$ that $E$ is in the isogeny class of $E_f$.

Under our hypotheses, this approach provides a deterministic and 
terminating algorithm to find all $E_f$. However, it can be extremely slow
 when $\n$ is small but the simplest curve in the isogeny class 
of $E_f$ has large coefficients.  For example, using this search method it would be
infeasible to find the curve \eqref{equation:fisher-curve} computed by Fisher using the
visibility of \Sha$[7]$.

\subsection{Sieved enumeration}\label{sec:sieve}

A refinement to the approach discussed above uses the $a_\p$ values to impose
congruence conditions modulo $\p$ on $E$.
If $f$ is a newform with Hecke eigenvalues
$a_\p$, then $\#\tilde E_f(R/\p)=\bN(\p)+1-a_{\p}$.
Given $\p$ not dividing the level $\n$, we can find all elliptic curves
modulo $\p$ with the specified number of points, especially when $\bN(\p) +1 - a_\p$ has
few prime factors. We impose these congruence conditions at multiple primes $\p_i$,
use the Chinese Remainder Theorem, and lift the resulting elliptic curves modulo $R/ (\prod \p_i)$
to non-singular elliptic curves over $R$.

While this method, like the previous one, will eventually terminate, it too is very
ineffective if every $E$ in the class of isogenous elliptic curves corresponding to $f$ has
large coefficients. However in practice, by optimally choosing the number of primes $\p_i$, 
a reasonably efficient implementation of this method can be obtained.

\subsection{Torsion families}\label{sec:torsion}
We find elliptic curves of small conductor by specializing explicit parametrizations
of families of elliptic curves over $F$ having specified torsion subgroups.
We use the parametrizations of \cite{kubert:torfam}.

\begin{theorem}[Kamienny-Najman, \cite{kamienny-najman}] The following is
a complete list of torsion structures for elliptic curves over $F$:

$$\begin{array}{lll}
\mathbb{Z}/m\mathbb{Z},   &1 \leq m \leq 10,& m = 12,\\
\mathbb{Z}/2\mathbb{Z} \oplus \mathbb{Z}/2m\mathbb{Z}, &  1 \leq m \leq 4,&\\
\mathbb{Z}/15\mathbb{Z}.&&
\end{array}$$ 
Moreover, there is a unique elliptic curve with $15$-torsion.
\end{theorem}

We use the following proposition to determine in which family to
search.
\begin{proposition}\label{prop:ptor}
Let $\ell$ be a prime and $E$ an elliptic curve over $F$.
Then $\ell \mid \# E'(F)_{\tor}$ for some elliptic curve $E'$ in the isogeny class
of $E$ if and only if $\ell \mid \bN(\p)+1 - a_{\p}$ for all
odd primes $\p$ at which $E$ has good reduction.
\end{proposition}
\begin{proof}
If $\ell  \mid \# E'(F)_{\tor}$, from the injectivity of the
reduction map at good primes \cite[Appendix]{katz:torsion}, we have that
$\ell \mid \#\tilde E'(\F_\p) = \bN(\p)+1 - a_{\p}$. The converse statement
is one of the main results of \cite{katz:torsion}.
\end{proof}

By applying Proposition~\ref{prop:ptor} for all $a_{\p}$ with $\p$ up
to some bound, we can decide whether or not it is {\em likely} that
some elliptic curve in the isogeny class of $E$ contains an $F$-rational
$\ell$-torsion point. If this is the case, then we search over those
families of elliptic curves with rational $\ell$-torsion. With a relatively small
search space, we thus find many elliptic curves with large coefficients more
quickly than with the algorithm of Section~\ref{sec:naive}.
For example, we first found the elliptic curve $E$ given by
$$y^2 + \vphi{}y = x^3 + \left(27 \vphi{} - 43\right)x + \left(-80 \vphi{} + 128\right) 
$$ with norm conductor $145$ by searching for elliptic curves with torsion subgroup $\Z/7\Z$.

\subsection{Congruence families}\label{sec:congfam}

Suppose that we are searching for an elliptic curve $E$ and we already know 
another elliptic curve $E'$ with $E[\ell] \approx E'[\ell]$, where $\ell$ is some prime
and $E[\ell]$ is irreducible. Twists of the modular curve $X(\ell)$ parametrize 
pairs of elliptic curves with isomorphic $\ell$-torsion subgroups, so finding 
rational points on the correct twist allows us to find curves with the same
mod $\ell$ Galois representation as $E'$.
Using this idea, we found the curve $E$ given by
\begin{multline}\label{equation:fisher-curve}
y^2 + \varphi{}xy =  x^3 + \left(\varphi{} - 1\right)x^2 + \\ 
\left(-257364 \varphi{} - 159063\right)x + \left(-75257037 \varphi{} - 46511406\right)
\end{multline}
with conductor $-6 \varphi + 42$, which has norm $1476$. Just given the $a_{\p}$,
we noticed that $E[7]\approx E'[7]$, where $E'$ has norm conductor $369$. The 
curve $E'$ had already been found via naive search, since it is given by the 
equation $y^2 + \left(\varphi{} + 1\right)y = x^3 + \left(\varphi{} - 1\right)x^2 + 
\left(-2 \varphi{}\right)x$. For any elliptic curve, the equation for the correct
twist of $X(7)$ was found both by Halberstadt and Kraus \cite{halb_kraus:XE7} and
by Fisher \cite{fisher:families_cong}, whose methods also yield formulas for the
appropriate twists of $X(9)$ and $X(11)$.

Fisher had already implemented \textsc{Magma} \cite{magma} routines to find
$\ell$-congruent elliptic curves over $\Q$ using these equations and was able
to modify his work for $\Q(\sqrt 5)$. Fortunately, our curve $E$ was then easily
found.


\subsection{Twisting}\label{sec:twisting}

Let $E$ be an elliptic curve over $F$. A \emph{twist} $E'$ of $E$ is an elliptic
curve over $F$ that is isomorphic to $E$ over some extension of $F$. A
\emph{quadratic twist} is a twist in which the extension has degree
$2$.  We can use twisting to find elliptic curves that may otherwise be
difficult to find as follows: starting with a known elliptic curve $E$
of some (small) conductor, we compute its twists of conductor up to
some bound, and add them to our table.

More explicitly, if $E$ is given by $y^2=x^3+ax+b$ and $d\in F^*$,
then the twist $E^d$ of $E$ by $d$ is given by $dy^2=x^3+ax+b$; in
particular, we may assume that $d$ is square free.  The following
is well known:
\begin{proposition}\label{twist:bound}
If $\n$ is the
conductor of $E$ and $d \in \mathcal{O}_F$ is non-zero, squarefree and coprime to $\n$, then the
conductor of $E^d$ is divisible by $d^2\n$.
\end{proposition}
\begin{proof}
There are choices of Weierstrass equations such that 
$\Delta(E^d) = 2^{12} d^6 \Delta(E)$, where $\Delta$
is the discriminant.
Thus the elliptic curve $E^d$ has bad reduction at each prime that divides $d$,
because twisting introduces a $6$th power of the squarefree $d$ into
the discriminant, and $d$ is coprime to $\Delta(E)$, so no change of
Weierstrass equation can remove this $6$th power.  Moreover, $E^d$ is
isomorphic to $E$ over an extension of the base field, so $E^d$ has
potentially good reduction at each prime dividing $d$.  Thus the
reduction at each prime dividing $d$ is additive.  The conductor is
unchanged at the primes dividing $\n$ because of the formula relating
the conductor, discriminant and reduction type (see \cite[App.~C,\S15]{silverman:aec}),
that formation of N\'eron models commutes with unramified base change,
and the fact that at the primes that divide $\n$ the minimal discriminant of $E^d$ is
the same as that of $E$.
%
\end{proof}

To find all twists $E^d$ with norm conductor at most $B$, we twist $E$
by all $d$ of the form $\pm \vphi^{\delta} d_0 d_1$, where $\delta\in
\{0,1\}$, $d_0$ is a product of a fixed choice of generators for
the prime divisors of $\n$, 
$d_1$ is a squarefree product of a fixed choice of generators of 
primes not dividing $\n$, and
$|\bN(d_1)| \leq \sqrt{B/C}$, where $C$ is the norm of the product of
the primes that exactly divide $\n$. We know from \ref{twist:bound} that 
this search is exhaustive.

For example, let $E$ be given by $y^2 + xy + \vphi{}y = x^3 +
\left(-\vphi{} - 1\right)x^2 $ of conductor $5\vphi - 3$ having norm $31$.
Following the above strategy to find twists of norm conductor $\leq
B := 1831$, we have $C=31$ and square-free $d_1$ such that
$|\bN(d_1)| \leq \sqrt{B/C} \approx 7.6\ldots$. Thus $d_1 \in \{1, 2,
\vphi, 2\vphi\}$ and checking all possibilities for
$\vphi^{\delta} d_0 d_1$, we find the elliptic curve $E^{-\vphi-2}$ having
norm conductor $775$ and the elliptic curve $E^{5\vphi-3}$ having norm conductor
$961$. Other twists have larger norm conductors, e.g., $E^2$ has norm
conductor $126976=2^{12}\cdot 31$.

\subsection{Elliptic Curves with good reduction outside $S$}\label{sec:cremona-lingham}

We use the algorithm of Cremona and Lingham from
\cite{cremona-lingham} to find all elliptic curves $E$ having good
reduction at primes outside of a finite set $\mathcal{S}$ of primes in
$F$. This algorithm has limitations over a general number
field $K$ due to the difficulty of finding a generating set for $E(K)$ and
points on $E$ defined over $\cO_K$.
Using Cremona's $\textsc{Magma}$ implementation of the algorithm, we 
found several elliptic curves not found by other methods, e.g.,
$y^2 + \left(\vphi{} +  1\right)xy + y = x^3 -x^2 + \left(-19 \vphi{} - 
 39\right)x + \left(-143 \vphi{} - 4\right),$ 
which has norm conductor $1331$.

\subsection{Special values of twisted $L$-series}\label{sec:specialvalues}
\newcommand{\Omegap}{\Omega^+}
\newcommand{\Omegam}{\Omega^-}
\newcommand{\Omegapp}{\Omega^{++}}
\newcommand{\Omegapm}{\Omega^{+-}}
\newcommand{\Omegamp}{\Omega^{-+}}
\newcommand{\Omegamm}{\Omega^{--}}
\newcommand{\OmegammEguess}{\Omega^{--}_{E, \mathrm{guess}}}
\newcommand{\OmegampEguess}{\Omega^{-+}_{E, \mathrm{guess}}}
\newcommand{\OmegapmEguess}{\Omega^{+-}_{E, \mathrm{guess}}}
\newcommand{\OmegappEguess}{\Omega^{++}_{E, \mathrm{guess}}}

In \cite{dembele:elliptic-curves-quadratic-fields}, Lassina \dembele
outlines some methods for finding modular elliptic curves from Hilbert
modular forms over real quadratic fields. Formally, these methods are
not proven to be any better than a direct search procedure, as they
involve making a large number of guesses, and a priori we do not know
just how many guesses we will need to make. And unlike other methods
described in this paper, this method requires many Hecke eigenvalues,
and computing these takes a lot of time. However, this method
certainly works extremely well in many cases, and after tuning it by
using large tables of elliptic curves that we had already computed, we are able
to use it to find more elliptic curves that we would have had no hope of
finding otherwise; we will give an example of one of these elliptic curves
later. 

When the level $\n$ is not square, \dembele's method relies on computing
or guessing periods of the elliptic curve by using special values of
$L$-functions of twists of the elliptic curve. In particular, the only inputs
required are the level of the Hilbert modular form and its $L$-series. So
we suppose that we know the level $\n = (N)$ of the form, where $N$ is
totally positive, and that we have sufficiently many coefficients of its
$L$-series $\ap{1}, \ap{2}, \ap{3}, \ldots$.

Let $\sigma_1$ and $\sigma_2$ denote the embeddings of $F$ into the
real numbers, with $\sigma_1(\varphi) \approx 1.61803\ldots$. For
an elliptic curve $E$ over $F$ we get two associated embeddings into
the complex numbers, and hence a pair of period lattices. Let
$\Omegap_E$ denote the smallest positive real period corresponding to the
embedding $\sigma_1$, and similarly define $\Omegam_E$ to be
the smallest period which lies on the positive imaginary axis. We
will refer to these as the periods of $E$, and as the period lattices
are interchanged when $E$ is replaced with its conjugate elliptic curve,
we let $\Omegap_{\Ebar}$ and $\Omegam_{\Ebar}$ denote the least real and
imaginary periods of the lattice under the embedding $\sigma_2$.

For ease, we write
\begin{align*}
\Omegapp_E &= \Omegap_E\Omegap_{\Ebar}& \ \Omegapm_E &= \Omegap_E\Omegam_{\Ebar} \\
\Omegamp_E &= \Omegam_E\Omegap_{\Ebar}& \ \Omegamm_E &= \Omegam_E\Omegam_{\Ebar}.
\end{align*}
We refer to these numbers as the {\em mixed periods} of $E$.

\subsubsection{Recovering the elliptic curve from its mixed periods}
If we know these mixed periods to sufficient precision, it is not hard
to recover the elliptic curve $E$. Without the knowledge of the discriminant
of the elliptic curve, we do not know the lattice type of the elliptic curve and its
conjugate, but there are only a few possibilities for what they might
be. This gives us a few possibilities for the $j$-invariant of
$E$. Observe that $\sigma_1(j(E))$ is either $j(\tau_1(E))$ or $j(\tau_2(E))$
and $\sigma_2(j(E))$ is either $j(\tau_1(\Ebar))$ or $j(\tau_2(\Ebar))$,
where
\begin{align*}
    \tau_1(E) &= \frac{\Omegamp_E}{\Omegapp_E} = \frac{\Omegam_E}{\Omegap_E}&  \tau_2(E) &= \frac{1}{2}\left(1 + \frac{\Omegamp_E}{\Omegapp_E}\right) = \frac{1}{2}\left(1 + \frac{\Omegam_E}{\Omegap_E}\right) \\
    \tau_1(\Ebar) &= \frac{\Omegapm_E}{\Omegapp_E} = \frac{\Omegam_E}{\Omegap_E} &    \tau_2(\Ebar) &= \frac{1}{2}\left(1 + \frac{\Omegapm_E}{\Omegapp_E}\right) = \frac{1}{2}\left(1 + \frac{\Omegam_{\Ebar}}{\Omegap_{\Ebar}}\right)
\end{align*}
and $j(\tau)$ is the familiar
\[
    j(\tau) = e^{-2\pi i \tau} + 744 + 196884e^{2\pi i \tau} + 21493760e^{4\pi i \tau} + \cdots.
\]
We try each pair of possible embeddings for $j(E)$ in turn, and recognize
possibilities for $j(E)$ as an algebraic number. We then construct elliptic curves $E'$
corresponding to each possibility for $j(E)$. By computing a few $\ap{}(E)$, we
should be able to determine whether we have chosen the correct $j$-invariant,
in which case $E'$ will be a twist of $E$. We can then recognize which twist it is
in order to recover $E$.

In practice, of course, as we have limited precision, and as $j(E)$ will not be
an algebraic integer, it may not be feasible to directly determine its exact value,
especially if its denominator is large.

To get around the problem of limited precision, we suppose that we have some extra information; 
namely, the discriminant $\Delta_E$ of the elliptic curve we are looking for.  With $\Delta_E$ in
hand we can directly determine which $\tau$ to choose: if $\sigma_1(\Delta_E) > 0$ then
$\sigma_1(j(E)) = j(\tau_1(E))$, and if $\sigma_1(\Delta_E) < 0$ then $\sigma_1(j(E)) = j(\tau_2(E))$,
and similarly for $\sigma_2$. We then compute
$\sigma_1(c_4(E)) = (j(\tau) \sigma_1(\Delta_E))^{1/3}$
and $\sigma_2(c_4(E)) = (j(\tau') \sigma_2(\Delta_E))^{1/3}$.

Using the approximations of the two embeddings of $c_4$, we can recognize $c_4$ approximately
as an algebraic integer. Specifically, we compute
\[
    \alpha = \frac{\sigma_1(c_4) + \sigma_2(c_4)}{2}\ \ \  \text{and} \ \ \ 
    \beta = \frac{\sigma_1(c_4) - \sigma_2(c_4)}{2\sqrt{5}}.
\]
Then $c_4 = \alpha + \beta\sqrt{5}$, and we can find $c_6$.

\newcommand{\Deltaguess}{\Delta_{\textrm{guess}}}
\newcommand{\cfourguess}{c_{4,\textrm{guess}}}
\newcommand{\csixguess}{c_{6,\textrm{guess}}}
\newcommand{\Eguess}{E_\textrm{guess}}

In practice, there are two important difficulties we must overcome: we do not know
$\Delta_E$ and it may be quite difficult to get high precision approximations to the
mixed periods, and thus we may not be able to easily compute $c_4$. Thus, we actually
proceed by choosing a $\Deltaguess$ from which we
compute half-integers $\alpha$ and $\beta$ and an integer
$a + b\varphi \approx \alpha + \beta\sqrt5$, arbitrarily rounding either $a$ or $b$ if necessary.
We then make some choice of search range $M$, and for each pair of integers $m$ and~$n$,
bounded in absolute value by $M$, we try each $\cfourguess = (a + m) + (b + n)\varphi$.

Given $\cfourguess$, we attempt to solve
\[
    \csixguess = \pm \sqrt{ \cfourguess^3 - 1728 \Deltaguess },
\]
and, if we can, we use these to construct a elliptic curve $\Eguess$. If $\Eguess$ has
the correct conductor and the correct Hecke eigenvalues, we declare that
we have found the correct elliptic curve; otherwise, we proceed to the next guess.

For a choice of $\Deltaguess$, we will generally start with the conductor $N_E$,
and then continue by trying unit multiples and by adding in powers of factors of $N_E$.

\subsubsection{Guessing the mixed periods}
We have thus far ignored the issue of actually finding the mixed periods of the elliptic curve that
we are looking for. Finding them presents an extra difficulty as our procedure involves
even more guesswork. \dembele's idea is to use special values of twists of the $L$-function
$L(f, s)$. Specifically, we twist by primitive quadratic Dirichlet characters over $\OF$,
which are homomorphisms $\chi : (\OF/\fc)^* \rightarrow \pm1$, pulled back to $\OF$.

In the case of odd prime conductor, which we will stick to here, there is
just a single primitive quadratic character, which is the quadratic residue symbol.
A simple way to compute it is by making a
table of squares, or by choosing a primitive root of $g \in (\OF/\fc)^*$, assigning
$\chi(g) = -1$, and again making a table by extending multiplicatively. Alternatively,
one could use a reciprocity formula as described in \cite{boylan-skoruppa:hecke-gauss-sums}.
For general conductor, one can compute with products of characters having prime conductor.

For a given $f$ and a primitive $\chi$, we can construct the twisted $L$-function
\[
    L(f, \chi, s) = \sum_{\m \subseteq \OF} \frac{\chi(m)a{_{\m}}}{N(\m)^s},
\]
where $m$ is a totally positive generator of $\m$. (Note that $\chi$ is not well defined
on ideals, but {\em is} well defined on totally positive generators of ideals.) $L(f, \chi, s)$
will satisfy a functional equation similar to that of $L(f, s)$, but the conductor is
multiplied by $\Norm(\fc)^2$ and the sign is multiplied by $\chi(-N)$.

Oda \cite{oda:periods} conjectured relations between the periods of $f$ and the
associated elliptic curve $E$ and gave some relations between the periods of $f$ and central
values of $L(s, \chi, 1)$. Stronger versions of these relations are conjectured, and
they are what \dembele uses to obtain information about the mixed periods of $E$.
Specifically, \dembele distills the following conjecture
from \cite{bertolini-darmon-green}, which we further simplify to state
specifically for $\Q(\sqrt5)$.
\begin{conjecture}
If $\chi$ is a primitive quadratic character with conductor $\fc$ relatively prime to the
conductor of $E$, with $\chi(\varphi) = s'$ and
$\chi(1 - \varphi) = s$, (where $s, s' \in \{+, -\} = \{\pm1\}$), then
\[
    \Omega^{s,s'}_E = c_\chi \tau(\overline\chi) L(E, \chi, 1)\sqrt5,
\]
for some integer $c_\chi$, where $\tau(\chi)$ is the Gauss sum
\[
    \tau(\chi) = \sum_{\alpha \bmod \fc} \chi(\alpha) \exp\left(2 \pi i \Tr\left(\alpha/m\sqrt{5}\right) \right),
\]
with $m$ a totally positive generator of $\fc$.
\end{conjecture}
\begin{remark}
  The Gauss sum is more innocuous than it
  seems. For odd conductor $\fc$ it is of size $\sqrt{\Norm(\fc)}$, while
  for an even conductor it is of size $\sqrt{2 \Norm(\fc) }$.  Its sign is
  a $4$-th root of unity, and whether it is real or imaginary can be
  deduced directly from the conjecture, as it matches with the sign of
  $\Omega^{s,s'}_E$. In particular, $\tau(\chi)$ is real when
  $\chi(-1) = 1$ and imaginary when $\chi(-1) = -1$, which is a
  condition on $\Norm(\fc) \bmod 4$, as $\chi(-1) \equiv \Norm(\fc) \pmod
  4$. This can all be deduced, for example, from
  \cite{boylan-skoruppa:hecke-gauss-sums}.

Also, note that \dembele writes this conjecture with an additional factor of $4\pi^2$;
this factor does not occur with the definition of $L(f, s)$ that we have given.

\end{remark}

\begin{remark}
Contained in this conjecture is the obstruction to carrying out the method described here when 
$\n$ is a square. If the sign of the functional equation of $L(f, s)$ is $\epsilon_f$,
then the sign of $L(f, \chi, s)$ will be $\chi(-N)\epsilon_f$.
When $\n$ is a perfect square, this is completely determined by whether
or not $\chi(\varphi) = \chi(1 - \varphi)$, so we can only obtain information about
either $\Omegamm$ and $\Omegapp$ or $\Omegamp$ and $\Omegapm$, and we need three of these
values to find $E$.
\end{remark}

With this conjecture in place, we can describe a method for guessing the mixed periods of
$E$.
Now, to proceed, we construct four lists of characters
up to some conductor bound $M$ (we are restricting to odd prime modulus here for simplicity,
as primitivity is ensured, but this is not necessary):
\[
    S^{s,s'} = \{ \chi \bmod \p : \chi(\phi) = s', \chi(1 - \phi) = s,
            (\p, \n) = 1, \Norm(\p) < M, \chi(-N) = \epsilon_f\}.
\]
Here $s, s' \in \{+, -\} = \{ \pm 1 \}$ again, and we restrict our choice of characters
to force the functional equation of $L(s, \chi, f)$ to have positive sign so that there
is a good chance that it does not vanish as the central point.
We will consider these lists to be ordered by the
norms of the conductors of the characters in increasing order, and index their elements as
$\chi^{s,s'}_0, \chi^{s,s'}_1, \chi^{s,s'}_2, \ldots$. For each character we compute the central
value of the twisted $L$-function to get four new lists
\[
    \mathcal{L}^{s,s'} = \{ i^{ss'}\sqrt{5 \Norm(\p) } L(E,\chi,1), \chi \in S^{s,s'}\} =
        \{\mathcal{L}^{s,s'}_0, \mathcal{L}^{s,s'}_1, \ldots\}.
\]
These numbers should now all be integer multiples of the mixed periods, so to get an idea
of which integer multiples they might be, we compute each of the ratios 
\[
    \frac{\mathcal{L}^{s,s'}_0}{\mathcal{L}^{s,s'}_k} = \frac{c_{\chi^{s,s'}_0}}{c_{\chi^{s,s'}_k}} \in \Q,\quad  k = 1, 2, \ldots,
\]
attempt to recognize these as rational numbers, and 
choose as an initial guess
\[
    \Omega^{ss'}_{E, \mathrm{guess}} = \mathcal{L}^{s,s'}_0\left(\mathrm{lcm}\left\{ \mathrm{numerator}\left(\frac{\mathcal{L}^{s,s'}_0}{\mathcal{L}^{s,s'}_k}\right), k = 1,2, \ldots \right\}\right)^{-1}.
\]

\subsubsection{An example}
We give an example of an elliptic curve that we were only able to find by using
this method. At level $\n = (-38\varphi + 26)$ we found a newform $f$, computed
\begin{multline*}
    a_{(2)}(f) = -1,\  a_{(-2\varphi + 1)}(f) = 1,\  a_{(3)}(f) = -1, \\ 
    a_{(-3\varphi + 1)}(f) = -1,\  a_{(-3\vphi + 2)}(f) = -6, \cdots, a_{(200\varphi - 101)}(f) = 168
\end{multline*}
and determined, by examining the $L$-function, that the sign of the functional equation
should be $-1$. (In fact, we do not really need to know the sign of the functional equation,
as we would quickly determine that $+1$ is wrong when attempting to find the mixed periods.)
Computing the sets of characters described above, and choosing the first $3$ of each, we
have
\begin{multline*}
S^{--} = \{\chi_{(\varphi + 6)}, \chi_{(7)}, \chi_{(7\varphi - 4)}\}, \ \ 
    S^{-+} = \{\chi_{(-3\varphi + 1)}, \chi_{(5 \varphi - 2)}, \chi_{(\varphi - 9)}\} \\
S^{+-} = \{\chi_{(-4\varphi + 3)}, \chi_{(5\varphi - 3)}, \chi_{(-2\varphi + 13)} \} \ \ 
    S^{++} = \{\chi_{(\varphi + 9)}, \chi_{(9\varphi - 5)}, \chi_{(\varphi + 13)} \}.
\end{multline*}

By using the $5133$ eigenvalues above as input to
Rubinstein's {\tt lcalc} \cite{lcalc}, we compute the lists of approximate
values
\[
\begin{split}
\mathcal{L}^{--} &= \{-33.5784397862407, -3.73093775400387, -18.6546887691646 \} \\ 
\mathcal{L}^{-+} &= \{18.2648617736017i, 32.8767511924831i, 3.65297235421633i \} \\
\mathcal{L}^{+-} &= \{41.4805656925342i, 8.29611313850694i, 41.4805677827298i\} \\
\mathcal{L}^{++} &= \{32.4909970742969, 162.454985515474, 162.454973589303\}.
\end{split}
\]
Note that {\tt lcalc} will warn us that we do not have enough coefficients to obtain good
accuracy, and we make no claim as far as the accuracy of these values is concerned.
Hoping that the ends will justify the means, we proceed forward.

Dividing each list by the first entry, and recognizing the quotients as rational
numbers, we get the lists
\[
\begin{split}
       \{1.000, 9.00000000005519, 1.80000000009351\}  &\approx \{1, 9, 9/5\}\\
       \{1.000, 0.555555555555555, 5.00000000068986 \} &\approx \{1, 5/9, 5\}\\
       \{1.000, 4.99999999999994, 0.999999949610245 \}  &\approx \{1, 5, 1\} \\ 
       \{1.000, 0.199999999822733, 0.200000014505165 \} &\approx \{1, 1/5, 1/5\},
\end{split}
\]
which may give an indication of the accuracy of our values. We now proceed with
the guesses
\begin{align*}
\phantom{MMMMM} \OmegammEguess &\approx -33.5784397862407/9 &\approx&\  -3.73093775402141 \phantom{MMMMM} \\
    \OmegampEguess &\approx 18.2648617736017i/5 &\approx&\  3.65297235472034i \\
    \OmegapmEguess &\approx 41.4805656925342i/5 &\approx&\  8.29611313850683i \\
    \OmegappEguess &\approx 32.4909970742969    &=&\  32.4909970742969.
\end{align*}
These cannot possibly be all correct, as $\Omegamm_E \Omegapp_E = \Omegamp_E \Omegapm_E$.
Still, we can choose any three and get a reasonable guess, and in fact we may choose all
possible triples, dividing some of the guesses by small rational numbers, and choosing the
fourth guess to be consistent with the first three; we build a list of possible embeddings
of $j(E)$, which will contain the possibility $\sigma_1(j(E)) \approx 1.365554233954 \times 10^{12}$,
$\sigma_2(j(E)) \approx 221270.95861123$, which is a possibility if 
\[\Omegamp_E = \OmegampEguess,\ \ 
\Omegapm_E = \OmegapmEguess,\ \ \Omegamp_E = \frac{\OmegampEguess}{2},\ \ \Omegapp_E = \frac{\OmegappEguess}{8}.\]
Cycling through many discriminants, we eventually try
\[
    \Deltaguess = \varphi \cdot 2^5 \cdot (19\varphi - 13),
\]
which leads us to the guess
\begin{align*}
    \sigma_1(\cfourguess) &= (\sigma_1(j(E)) \sigma_1(\Deltaguess))^{1/3} \approx 107850.372979378 \\
    \sigma_2(\cfourguess) &= (\sigma_2(j(E)) \sigma_2(\Deltaguess))^{1/3} \approx 476.625892034286.
\end{align*}
We have enough precision to easily recognize this as
\[
    \cfourguess = \frac{108327 + 48019 \sqrt5}{2} = 48019\varphi + 30154,
\]
and
\[
    \sqrt{\cfourguess^3 - 1728\Deltaguess}
\]
does in fact have two square roots: $\pm(15835084\varphi + 9796985)$. We try both of them, and
the choice with the minus sign gives the elliptic curve
\[
y^2 + \varphi xy + \varphi y = x^3 + \left(\varphi - 1\right)x^2 + \left(-1001 \varphi - 628\right)x + \left(17899 \varphi + 11079\right),
\]
which has the correct conductor. We compute a few values of $a_\p$ for this elliptic curve, and it turns
out to be the one that we are looking for.

\section{Enumerating the elliptic curves in an isogeny class}\label{sec:isoclass}

Given an elliptic curve $E/F$, we wish to find representatives up to isomorphism
for all elliptic curves $E'/F$ that are isogenous
to $E$ via an isogeny defined over $F$. The analogue of this problem
over $\Q$ has an algorithmic solution as explained in
\cite[\S3.8]{cremona:algs}; it relies on:
\begin{enumerate}
\item  Mazur's theorem \cite{mazur:rational} 
that if $\psi:E\to E'$ is a $\Q$-rational isogeny of prime degree, 
then $\deg(\psi)\leq 163$.
\item Formulas of V\'elu \cite{velu:isogenies} 
that provide a way to explicitly enumerate all $p$-isogenies (if any) with domain $E$.  
V\'elu's formulas are valid for any number field, but so far there has not been an explicit
generalization of Mazur's theorem for any number field other than $\Q$. 
\end{enumerate}

\begin{remark}
  Assume the generalized Riemann hypothesis.  Then work of Larson-Vaintrob 
  from \cite{larson-vaintrob} implies that there is an effectively computable constant $C_F$ 
  such that if $\varphi: E \to E'$ is a prime-degree isogeny defined over $F$ and
  $E'$ and $E$ are not isomorphic over $F$, then $\varphi$ has degree at most $C_F$.
\end{remark}

Since we are interested in specific isogeny classes, 
we can use the algorithm described in \cite{billerey:isog} that takes as
input a specific non-CM elliptic curve $E$ over a number field $K$, and
outputs a provably finite list of primes $p$ such that $E$ might have a
$p$-isogeny. The algorithm is particularly easy to implement in the
case when $K$ is a quadratic field, as explained in
\cite[\S2.3.4]{billerey:isog}.  Using this algorithm combined with
V\'elu's formulas, we were able to enumerate {\em all}
isomorphism classes of elliptic curves isogenous to the elliptic curves we found via the
methods of Section~\ref{sec:finding}, and thus divide our isogeny classes into
isomorphism classes.

\section{CM elliptic curves over $F$}\label{sec:cm}

In this section we make some general remarks about CM elliptic curves
over $F$.  The main surprise is that there are $31$ distinct
$\Qbar$-isomorphism classes of CM elliptic curves defined over $F$,
more than for any other quadratic field.

\begin{proposition}
The field $F$ has more isomorphism classes of CM
elliptic curves than any other quadratic field.
\end{proposition}
\begin{proof}
Let $K$ be a quadratic extension of $\Q$.
Let $H_D$ denote the Hilbert class polynomial of the CM order
$\cO_D$ of discriminant $D$, so $H_D\in \Q[X]$ is the minimal
polynomial of the $j$-invariant $j_D$ of any elliptic curve $E=E_D$ with
CM by $\cO_D$.   Since $K$ is Galois, we have $j_D \in K$ if and only if
$H_D$ is either linear or quadratic with both roots in $K$.
The $D$ for which $H_D$ is linear are the thirteen values
$-3, -4, -7, -8, -11, -12, -16, -19, -27, -28, -43, -67, -163$.
According to \cite{cremona:abvar}, the
$D$ for which $H_D$ is quadratic are the following $29$
discriminants:
\begin{align*}
&-15, -20, -24, -32, -35, -36, -40, -48, -51, -52, -60, \\
&-64, -72, -75, -88, -91, -99, -100, -112, -115, -123, \\
&-147, -148, -187, -232, -235, -267, -403, -427.
\end{align*}
 
By computing discriminants of these Hilbert class polynomials,
we obtain the following table:


\begin{center}
\begin{tabular}{@{}llcll@{}}\toprule
Field & $D$ so $H_D$ has roots in field & \phantom{ab} & Field & $D$ so $H_D$ has roots in field \\\cmidrule{1-2}\cmidrule{4-5}
$\Q(\sqrt{2})$ & $-24,-32,-64,-88$ & & $\Q(\sqrt{21})$ & $-147$ \\
$\Q(\sqrt{3})$ & $-36,-48$ & & $\Q(\sqrt{29})$ & $-232$ \\
\multirow{2}{*}{$\Q(\sqrt{5})$} & $-15,-20,-35,-40,-60,$ & & $\Q(\sqrt{33})$ & $-99$ \\
      & $-75,-100,-115,-235$ & & $\Q(\sqrt{37})$ & $-148$ \\
$\Q(\sqrt{6})$ & $-72$ & & $\Q(\sqrt{41})$ & $-123$ \\
$\Q(\sqrt{7})$ & $-112$ & & $\Q(\sqrt{61})$ & $-427$ \\
$\Q(\sqrt{13})$ & $-52,-91,-403$ & & $\Q(\sqrt{89})$ & $-267$ \\
$\Q(\sqrt{17})$ & $-51,-187$ & & & \\\bottomrule
\end{tabular}
\end{center}
The claim follows because the $\Q(\sqrt{5})$ row is largest,
containing $9$ entries.  There are thus $31 = 2\cdot 9 + 13$ distinct CM $j$-invariants
in $\Q(\sqrt{5})$.

\end{proof}

\section{Related future projects}\label{sec:future}

It would be natural to extend the tables to the first known elliptic curve of
rank $3$ over $F$, which may be the elliptic curve $y^2 + y = x^3 -2x + 1$ of
norm conductor $163^2=26569$.  It would also be interesting to make a
table in the style of \cite{stein-watkins:ants5}, and compute analytic
ranks of the large number of elliptic curves that we would find; this would
benefit from Sutherland's {\tt smalljac} program, which has very fast code
for computing $L$-series coefficients.  Some aspects of the tables
could also be generalized to modular abelian varieties $A_f$ attached
to Hilbert modular newforms with not-necessarily-rational Hecke
eigenvalues; in particular, we could enumerate the $A_f$ up to some
norm conductor, and numerically compute their analytic ranks.

\section{Tables}\label{sec:tables}

As explained in Sections~\ref{sec:finding} and \ref{sec:isoclass},
assuming Conjecture~\ref{conj:mod}, we found the complete list of
elliptic curves with norm conductor up to $1831$, which is the first
norm conductor of a rank $2$ elliptic curve over $F$. The complete
dataset can be downloaded from \cite{sqrt5data}.

In each of the following tables \#isom refers to the number of isomorphism
classes of elliptic curves, \#isog refers to the number of isogeny classes of
elliptic curves, $\n$ refers to the conductor of the given elliptic curve,
and Weierstrass equations are given in the form $[a_1,a_2,a_3,a_4,a_6]$.

Table~\ref{table:total-counts} gives the number of elliptic curves and isogeny
classes we found. Note that in these counts we do not exclude
conjugate elliptic curves, i.e., if $\sigma$ denotes the nontrivial element of
$\Gal(F/\Q)$, then we count $E$ and $E^{\sigma}$ separately if they
are not isomorphic.
\begin{center}
\begin{table}[h]
\caption{Elliptic Curves over $\Q(\sqrt{5})$\label{table:total-counts}}
\begin{tabular}{@{}lcrrcr@{}}\toprule
\textbf{rank} & \phantom{a} &\textbf{\#isog} & \textbf{\#isom} & \phantom{a} & \textbf{smallest $\Norm(\n)$} \\\cmidrule{1-1}\cmidrule{3-4}\cmidrule{6-6}
$0$   & & $745$  & $2174$ & & $31$\\
$1$   & & $667$  & $1192$ & & $199$ \\
$2$   & & $2$    & $2$    & & $1831$ \\\cmidrule{1-6}
total & & $1414$ & $3368$ & & -\; \\\bottomrule
\end{tabular}
\end{table}
\end{center}

Table~\ref{table:isogeny-sizes} gives counts of the number of isogeny
classes of elliptic curves in our data of each size; note that we find some
isogeny classes of cardinality $10$, which is bigger than what one
observes with elliptic curves over $\QQ$.
\begin{center}
\begin{table}[h]
\caption{Number of isogeny classes of a given size\label{table:isogeny-sizes}}
\begin{tabular}{@{}lcrrrrrrrcr@{}}\toprule
& \phantom{a} & \multicolumn{7}{c}{\textbf{size}} & \phantom{a} & \\\cmidrule{3-9}
\textbf{bound} & & 1   & 2   & 3  & 4   & 6  & 8  & 10 & & \textbf{total} \\\midrule
199  & & 2   & 21  & 3  & 20  & 8  & 9  & 1  & & 64    \\
1831 & & 498 & 530 & 36 & 243 & 66 & 38 & 3  & & 1414  \\\bottomrule
\end{tabular}
\end{table}
\end{center}

Table~\ref{table:increasing-counts} gives the number of elliptic curves and isogeny classes
up to a given norm conductor bound. Note that the first elliptic curve of rank $1$ has
norm conductor $199$, and there are no elliptic curves of norm conductor $200$.
\begin{center}
\begin{table}[h]
\caption{Counts of isogeny classes and elliptic curves with bounded norm conductors and specified ranks\label{table:increasing-counts}}
\begin{tabular}{@{}lcrrrcrcrrrcr@{}}\toprule
& \phantom{a} & \multicolumn{5}{c}{\textbf{\#isog}}                  & \phantom{ab} & \multicolumn{5}{c}{\textbf{\#isom}}                  \\\cmidrule{3-7}\cmidrule{9-13}
&             & \multicolumn{3}{c}{\textbf{rank}} & &                &              & \multicolumn{3}{c}{\textbf{rank}} & &                \\\cmidrule{3-5}\cmidrule{9-11}
\textbf{bound} &              & 0 & 1 & 2          & & \textbf{total} &              & 0 & 1 & 2                         & & \textbf{total} \\\midrule
200  & & 62  & 2   & 0 & & 64   & & 257  & 6    & 0 & & 263  \\
400  & & 151 & 32  & 0 & & 183  & & 580  & 59   & 0 & & 639  \\
600  & & 246 & 94  & 0 & & 340  & & 827  & 155  & 0 & & 982  \\
800  & & 334 & 172 & 0 & & 506  & & 1085 & 285  & 0 & & 1370 \\
1000 & & 395 & 237 & 0 & & 632  & & 1247 & 399  & 0 & & 1646 \\
1200 & & 492 & 321 & 0 & & 813  & & 1484 & 551  & 0 & & 2035 \\
1400 & & 574 & 411 & 0 & & 985  & & 1731 & 723  & 0 & & 2454 \\
1600 & & 669 & 531 & 0 & & 1200 & & 1970 & 972  & 0 & & 2942 \\
1800 & & 729 & 655 & 0 & & 1384 & & 2128 & 1178 & 0 & & 3306 \\
1831 & & 745 & 667 & 2 & & 1414 & & 2174 & 1192 & 2 & & 3368 \\\bottomrule
\end{tabular}
\end{table}
\end{center}

Table~\ref{table:degree} gives the number of elliptic curves and isogeny classes with
isogenies of each degree; note that we do not see all possible isogeny
degrees. For example, the elliptic curve $X_0(19)$ has rank 1 over $F$, so
there are infinitely many elliptic curves over $F$ with degree 19 isogenies
(unlike over $\Q$ where $X_0(19)$ has rank $0$).
We also give an example of an elliptic curve (that need not have minimal conductor)
with an isogeny of the given degree.
\begin{center}
\begin{table}[h]
\caption{Isogeny degrees\label{table:degree}}
\begin{tabular}{@{}lcrrclr@{}}\toprule
\textbf{degree} & \phantom{a} & \textbf{\#isog} & \textbf{\#isom} & \phantom{a} & \textbf{example curve} & $\Norm(\n)$ \\\cmidrule{1-1}\cmidrule{3-4}\cmidrule{6-7}
None & & 498 & 498  & & $[\varphi+1,1,1,0,0]$                               & 991  \\
2    & & 652 & 2298 & & $[\varphi,-\varphi+1,0,-4,3\varphi-5]$              & 99   \\
3    & & 289 & 950  & & $[\varphi,-\varphi,\varphi,-2\varphi-2,2\varphi+1]$ & 1004 \\
5    & & 65  & 158  & & $[1,0,0,-28,272]$                                   & 900  \\
7    & & 19  & 38   & & $[0,\varphi+1,\varphi+1,\varphi-1,-3\varphi-3]$     & 1025 \\\bottomrule
\end{tabular}
\end{table}
\end{center}

Table~\ref{table:torsion} gives the number of elliptic curves with each
torsion structure, along with an example of an elliptic curve (again, not necessarily with minimal conductor)
with that torsion structure.
\begin{center}
\begin{table}[h]
\caption{Torsion subgroups\label{table:torsion}}
\begin{tabular}{@{}lcrclr@{}}\toprule
\textbf{structure} & \phantom{a} & \textbf{\#isom} & \phantom{a} & \textbf{example curve} & $\Norm(\n)$ \\\cmidrule{1-1}\cmidrule{3-3}\cmidrule{5-6}
1                    & & 796  & & $[0,-1,1,-8,-7]$                      & 225 \\
$\Z/2\Z$             & & 1453 & & $[\vphi,-1,0,-\vphi-1,\vphi-3]$       & 164 \\
$\Z/3\Z$             & & 202  & & $[1,0,1,-1,-2]$                       & 100 \\
$\Z/4\Z$             & & 243  & & $[\vphi+1,\vphi-1,\vphi,0,0]$         & 79  \\
$\Z/2\Z\oplus\Z/2\Z$ & & 312  & & $[0,\vphi+1,0,\vphi,0]$               & 256 \\
$\Z/5\Z$             & & 56   & & $[1,1,1,22,-9]$                       & 100 \\
$\Z/6\Z$             & & 183  & & $[1,\vphi,1,\vphi-1,0]$               & 55  \\
$\Z/7\Z$             & & 13   & & $[0,\vphi-1,\vphi+1,0,-\vphi]$        & 41  \\
$\Z/8\Z$             & & 21   & & $[1,\vphi+1,\vphi,\vphi,0]$           & 31  \\
$\Z/2\Z\oplus\Z/4\Z$ & & 51   & & $[\vphi+1,0,0,-4,-3\vphi-2]$          & 99  \\
$\Z/9\Z$             & & 6    & & $[\vphi,-\vphi+1,1,-1,0]$             & 76  \\
$\Z/10\Z$            & & 12   & & $[\vphi+1,\vphi,\vphi,0,0]$           & 36  \\
$\Z/12\Z$            & & 6    & & $[\vphi,\vphi+1,0,2\vphi-3,-\vphi+2]$ & 220 \\
$\Z/2\Z\oplus\Z/6\Z$ & & 11   & & $[0,1,0,-1,0]$                        & 80  \\
$\Z/15\Z$            & & 1    & & $[1,1,1,-3,1]$                        & 100 \\
$\Z/2\Z\oplus\Z/8\Z$ & & 2    & & $[1,1,1,-5,2]$                        & 45  \\\bottomrule
\end{tabular}
\end{table}
\end{center}

We computed the invariants in the Birch and Swinnerton-Dyer conjecture
for our elliptic curves, and solved for the conjectural order of $\Sha$;
Table~\ref{table:sha} gives the number of elliptic curves in our data
having each order of $\Sha$ as well as an elliptic curve of minimal conductor
exhibiting each of these orders.
\begin{center}
\begin{table}[h]
\caption{$\Sha$\label{table:sha}}
\begin{tabular}{@{}lcrclr@{}}\toprule
\textbf{\#\Sha} & \phantom{a} & \textbf{\#isom} & \phantom{a} & \textbf{first elliptic curve having \#\Sha}& $\Norm(\n)$ \\\cmidrule{1-1}\cmidrule{3-3}\cmidrule{5-6}
1                   & & 3191                & & $[1,\vphi+1,\vphi,\vphi,0]$                        & 31                    \\
4                   & & 84                  & & $[1, 1, 1, -110, -880]$                            & 45                    \\
\multirow{2}{*}{9}  & & \multirow{2}{*}{43} & & $[\vphi+1,-\vphi,1,-54686\vphi-35336,$             & \multirow{2}{*}{76}   \\
                    & &                     & & \multicolumn{1}{r}{$-7490886\vphi-4653177]$}       &                       \\
\multirow{2}{*}{16} & & \multirow{2}{*}{16} & & $[1,\vphi,\vphi+1,-4976733\vphi-3075797,$          & \multirow{2}{*}{45}   \\
                    & &                     & & \multicolumn{1}{r}{$-6393196918\vphi-3951212998]$} &                       \\
25                  & & 2                   & & $[0, -1, 1, -7820, -263580]$                       & 121                   \\
\multirow{2}{*}{36} & & \multirow{2}{*}{2}  & & $[1,-\vphi+1,\vphi,1326667\vphi-2146665,$          & \multirow{2}{*}{1580} \\
                    & &                     & & \multicolumn{1}{r}{$880354255\vphi-1424443332]$}   &                       \\\bottomrule
\end{tabular}
\end{table}
\end{center}

\newcommand{\etalchar}[1]{$^{#1}$}
\providecommand{\bysame}{\leavevmode\hbox to3em{\hrulefill}\thinspace}
\providecommand{\MR}{\relax\ifhmode\unskip\space\fi MR }
\providecommand{\MRhref}[2]{%
  \href{http://www.ams.org/mathscinet-getitem?mr=#1}{#2}
}
\providecommand{\href}[2]{#2}


\begin{thebibliography}{BDKM{\etalchar{+}}12}

\bibitem[BCDT01]{breuil-conrad-diamond-taylor}
C.~Breuil, B.~Conrad, F.~Diamond, and R.~Taylor, \emph{On the modularity of
  elliptic curves over {$\mathbf{Q}$}: wild 3-adic exercises}, J. Amer. Math.
  Soc. \textbf{14} (2001), no.~4, 843--939 (electronic),
  \url{http://math.stanford.edu/~conrad/papers/tswfinal.pdf}. \MR{2002d:11058}

\bibitem[BCP97]{magma}
W.~Bosma, J.~Cannon, and C.~Playoust, \emph{The {M}agma algebra system. {I}.
  {T}he user language}, J. Symbolic Comput. \textbf{24} (1997), no.~3--4,
  235--265, Computational algebra and number theory (London, 1993). \MR{1 484
  478}

\bibitem[BDG04]{bertolini-darmon-green}
Massimo Bertolini, Henri Darmon, and Peter Green, \emph{Periods and points
  attached to quadratic algebras}, Heegner points and {R}ankin {$L$}-series,
  Math. Sci. Res. Inst. Publ., vol.~49, Cambridge Univ. Press, Cambridge, 2004,
  pp.~323--367. \MR{2083218 (2005e:11062)}

\bibitem[BDKM{\etalchar{+}}12]{sqrt5data}
Jon Bober, Alyson Deines, Ariah Klages-Mundt, Ben LeVeque, R.~Andrew Ohana,
  Ashwath Rabindranath, Paul Sharaba, and William Stein, \emph{A {D}atabase of
  {E}lliptic {C}urves over {$\Q(\sqrt{5})$}}, 2012,
  \url{http://wstein.org/papers/sqrt5}.

\bibitem[Bil11]{billerey:isog}
Nicolas Billerey, \emph{Crit\`eres d'irr\'eductibilit\'e pour les
  repr\'esentations des courbes elliptiques}, Int. J. Number Theory \textbf{7}
  (2011), no.~4, 1001--1032. \MR{2812649}

\bibitem[BK75]{antwerpiv}
B.\thinspace{}J. Birch and W.~Kuyk (eds.), \emph{Modular functions of one
  variable. {I}{V}}, Springer-Verlag, Berlin, 1975, Lecture Notes in
  Mathematics, Vol. 476.

\bibitem[BMSW07]{bmsw:bulletins}
Baur Bektemirov, Barry Mazur, William Stein, and Mark Watkins, \emph{Average
  ranks of elliptic curves: tension between data and conjecture}, Bull. Amer.
  Math. Soc. (N.S.) \textbf{44} (2007), no.~2, 233--254 (electronic).
  \MR{2291676}

\bibitem[BS10]{boylan-skoruppa:hecke-gauss-sums}
Hatice Boylan and Nils-Peter Skoruppa, \emph{Explicit formulas for {H}ecke
  {G}auss sums in quadratic number fields}, Abh. Math. Semin. Univ. Hambg.
  \textbf{80} (2010), no.~2, 213--226. \MR{2734687 (2012c:11163)}

\bibitem[CL07]{cremona-lingham}
J.\thinspace{}E. Cremona and M.\thinspace{}P. Lingham, \emph{Finding all
  elliptic curves with good reduction outside a given set of primes},
  Experiment. Math. \textbf{16} (2007), no.~3, 303--312. \MR{2367320
  (2008k:11057)}

\bibitem[Cre]{cremona:onlinetables}
J.\thinspace{}E. Cremona, \emph{{E}lliptic {C}urves {D}ata},
  \url{http://www.warwick.ac.uk/~masgaj/ftp/data/}.

\bibitem[Cre92]{cremona:abvar}
\bysame, \emph{Abelian varieties with extra twist, cusp forms, and elliptic
  curves over imaginary quadratic fields}, J. London Math. Soc. (2) \textbf{45}
  (1992), no.~3, 404--416. \MR{1180252 (93h:11056)}

\bibitem[Cre97]{cremona:algs}
\bysame, \emph{Algorithms for modular elliptic curves}, second ed., Cambridge
  University Press, Cambridge, 1997,
  \url{http://www.warwick.ac.uk/~masgaj/book/fulltext/}.

\bibitem[Dem05]{dembele:hilbert5}
Lassina Demb{\'e}l{\'e}, \emph{Explicit computations of {H}ilbert modular forms
  on {${\Bbb Q}(\sqrt{5})$}}, Experiment. Math. \textbf{14} (2005), no.~4,
  457--466. \MR{2193808}

\bibitem[Dem08]{dembele:elliptic-curves-quadratic-fields}
\bysame, \emph{An algorithm for modular elliptic curves over real quadratic
  fields}, Experiment. Math. \textbf{17} (2008), no.~4, 427--438. \MR{2484426
  (2010a:11119)}

\bibitem[Dok04]{dokchitser:lfun}
Tim Dokchitser, \emph{Computing special values of motivic {$L$}-functions},
  Experiment. Math. \textbf{13} (2004), no.~2, 137--149,
  \url{http://arxiv.org/abs/math/0207280}. \MR{2068888 (2005f:11128)}

\bibitem[Fis12]{fisher:families_cong}
Tom Fisher, \emph{On {F}amilies of $n$-congruent {E}lliptic {C}urves}, Preprint
  (2012).

\bibitem[GV11]{greenberg-voight:shimura}
Matthew Greenberg and John Voight, \emph{Computing systems of {H}ecke
  eigenvalues associated to {H}ilbert modular forms}, Math. Comp. \textbf{80}
  (2011), no.~274, 1071--1092,
  \url{http://www.cems.uvm.edu/~voight/articles/heckefun-021910.pdf}.
  \MR{2772112 (2012c:11103)}

\bibitem[GZ86]{gross-zagier}
B.~Gross and D.~Zagier, \emph{Heegner points and derivatives of
  \protect{${L}$}-series}, Invent. Math. \textbf{84} (1986), no.~2, 225--320,
  \url{http://wstein.org/papers/bib/Gross-Zagier_Heegner_points_and_derivatives_of_Lseries.pdf}.
  \MR{87j:11057}

\bibitem[HK03]{halb_kraus:XE7}
Emmanuel Halberstadt and Alain Kraus, \emph{Sur la courbe modulaire
  {$X_E(7)$}}, Experiment. Math. \textbf{12} (2003), no.~1, 27--40. \MR{2002672
  (2004m:11090)}

\bibitem[Kat81]{katz:torsion}
N.\thinspace{}M. Katz, \emph{Galois properties of torsion points on abelian
  varieties}, Invent. Math. \textbf{62} (1981), no.~3, 481--502. \MR{82d:14025}

\bibitem[KN12]{kamienny-najman}
Sheldon Kamienny and Filip Najman, \emph{Torsion groups of elliptic curves over
  quadratic fields}, Acta. Arith. \textbf{152} (2012), 291--305.

\bibitem[Kol91]{kolyvagin:mordellweil}
V.\thinspace{}A. Kolyvagin, \emph{On the {M}ordell-{W}eil group and the
  {S}hafarevich-{T}ate group of modular elliptic curves}, Proceedings of the
  International Congress of Mathematicians, Vol.\ I, II (Kyoto, 1990) (Tokyo),
  Math. Soc. Japan, 1991, pp.~429--436. \MR{93c:11046}

\bibitem[Kub76]{kubert:torfam}
Daniel~Sion Kubert, \emph{Universal bounds on the torsion of elliptic curves},
  Proceedings of the London Mathematical Society \textbf{s3-33} (1976), no.~2,
  193--237.

\bibitem[LV]{larson-vaintrob}
E.\ Larson and D.\ Vaintrob, \emph{Determinants of subquotients of {G}alois
  representations associated to abelian varieties}, arXiv:1110.0255.

\bibitem[Maz78]{mazur:rational}
B.~Mazur, \emph{Rational isogenies of prime degree (with an appendix by {D}.
  {G}oldfeld)}, Invent. Math. \textbf{44} (1978), no.~2, 129--162.

\bibitem[Oda82]{oda:periods}
Takayuki Oda, \emph{Periods of {H}ilbert modular surfaces}, Progress in
  Mathematics, vol.~19, Birkh\"auser Boston, Mass., 1982. \MR{670069
  (83k:10057)}

\bibitem[Rub11]{lcalc}
M.\thinspace{}O. Rubinstein, \emph{Lcalc}, 2011,
  \url{http://oto.math.uwaterloo.ca/~mrubinst/l_function_public/CODE/}.

\bibitem[S{\etalchar{+}}12]{sage}
W.\thinspace{}A. Stein et~al., \emph{{S}age {M}athematics {S}oftware ({V}ersion
  4.8)}, The Sage Development Team, 2012, \url{http://www.sagemath.org}.

\bibitem[Sil92]{silverman:aec}
J.\thinspace{}H. Silverman, \emph{The arithmetic of elliptic curves},
  Springer-Verlag, New York, 1992, Corrected reprint of the 1986 original.

\bibitem[SW02]{stein-watkins:ants5}
William Stein and Mark Watkins, \emph{A database of elliptic curves---first
  report}, Algorithmic number theory (Sydney, 2002), Lecture Notes in Comput.
  Sci., vol. 2369, Springer, Berlin, 2002, \url{http://wstein.org/ecdb},
  pp.~267--275. \MR{2041090 (2005h:11113)}

\bibitem[V{\'e}l71]{velu:isogenies}
Jacques V{\'e}lu, \emph{Isog\'enies entre courbes elliptiques}, C. R. Acad.
  Sci. Paris S\'er. A-B \textbf{273} (1971), A238--A241.

\bibitem[Wil95]{wiles:fermat}
A.\thinspace{}J. Wiles, \emph{Modular elliptic curves and \protect{F}ermat's
  last theorem}, Ann. of Math. (2) \textbf{141} (1995), no.~3, 443--551,
  \url{http://users.tpg.com.au/nanahcub/flt.pdf}.

\bibitem[Zha01]{zhang:heightsshimura}
Shou-Wu Zhang, \emph{Heights of {H}eegner points on {S}himura curves}, Ann. of
  Math. (2) \textbf{153} (2001), no.~1, 27--147. \MR{1826411 (2002g:11081)}

\end{thebibliography}
\end{document}